\documentclass[12pt,leqno]{article}
\usepackage[a4paper]{geometry}

\usepackage[all]{xy}\SelectTips{cm}{}

\usepackage{amssymb, eucal, amsmath, amsthm}

\usepackage{tikz}
\usepackage{float}

\usepackage{mathptmx}
\usepackage[scaled=.90]{helvet}
\usepackage{courier}

\newtheorem{theorem}{Theorem}
\newtheorem{proposition}[theorem]{Proposition}

\newtheorem{corollary}[theorem]{Corollary}

\theoremstyle{definition}
\newtheorem{definition}[theorem]{Definition}
\newtheorem{example}[theorem]{Example}
\newtheorem{examples}[theorem]{Examples}
\newtheorem{remark}[theorem]{Remark}

\numberwithin{theorem}{section}
\numberwithin{equation}{section}

\newcommand{\RR}{\mathbb{R}}

\newcommand{\id}{\mathrm{id}}
\newcommand{\pr}{\mathrm{pr}}
\newcommand{\ev}{\mathrm{ev}}

\renewcommand{\emptyset}{\text{\O}}

\usepackage[pdftex]{hyperref} %, hyperindex=true, backref=page

\hypersetup{
  colorlinks = true,
  linkcolor = black,
  urlcolor = blue,
  citecolor = black,
}

\begin{document}

\title{Homotopies and the universal fixed point property}

\author{Markus Szymik}

\date{October 2013}

\maketitle

\renewcommand\abstractname{}

\begin{abstract}
\noindent
A topological space has the fixed point property if every continuous self-map of that space has at least one fixed point. We demonstrate that there are serious restraints imposed by the requirement that there be a choice of fixed points that is continuous whenever the self-map varies continuously. To even specify the problem, we introduce the universal fixed point property. Our results apply in particular to the analysis of convex subspaces of Banach spaces, to the topology of finite-dimensional manifolds and~CW complexes, and to the combinatorics of Kolmogorov spaces associated with finite posets.
\newline\phantom{ }\newline
%\newline\phantom{ }\newline
%Keywords: fixed point property, homotopy
\end{abstract}

%%%

Every continuous self-map of a non-empty, compact, convex subspace of a topological vector space has at least one fixed point. This theorem of Brouwer-Schauder-Tychonoff is a milestone of fixed point theory, and the applications of it are numerous. Whenever the existence of fixed points is guaranteed for every self-map, the space is called a fixed point space. This property is shared by many other spaces besides the ones already covered by that famous result.

But, what if the self-map varies continuously? Is there also a continuous choice of fixed points? This is the question that we address here. It is surprising that there is no previous literature on this topic in the general topological setting, despite the property being clearly desirable: in practice, it is genuinely less useful to know the existence of a solution to a problem, if this solution does not depend continuously on the initial conditions. The Banach fixed point theorem provides a class of examples where we do know that the fixed points depend continuously on the initial conditions (Example~\ref{ex:Banach}). However, the result evidently applies to metric spaces and contractions only. 

Additional requirements on fixed point spaces that guarantee the existence of continuous choices of fixed points in general are made precise here and lead up to the notion of {\em universal} fixed point spaces~(Definition~\ref{def:ufpp}). Among other things, we give a characterization of these in terms of mapping spaces~(Theorem~\ref{thm:criterion}) which is also the basis for the discussion of some of our later examples. We show that the universal fixed point property is rather strong. The only convex subspaces of Banach spaces which are universal fixed point spaces are singletons~(Theorem~\ref{thm:Banach_homotopy}), and similarly for finite-dimensional manifolds and~CW complexes~(Corollary~\ref{cor:manifolds_and_CW_homotopy}). Nonetheless, many examples of universal fixed point spaces other than singletons do exist, and some of them are not even contractible. For example, there are Kolmogorov spaces associated with finite posets that have these properties~(Example~\ref{ex:not_contractible}).

Unless otherwise stated, all vector spaces will be over the scalar field~$\RR$ of real numbers.

%%%

\section{A brief review of the fixed point property}\label{sec:fpp}

In this section, we recall the basic definitions and some immediate consequences that we need.

\begin{definition}{\bf (Fixed point property)}
A topological space~$X$ has the {\it fixed point property}, if every continuous self-map~$f\colon X\to X$ has a fixed point. One also says that~$X$ is a {\it fixed point space}.
\end{definition}

\begin{example}
By Brouwer's theorem~\cite{Brouwer}, all disks~$D^n=\{x\in\RR^n|\|x\|\leqslant1\}$ are fixed point spaces. 
\end{example}

The fixed point property has a long history, see~\cite[p.120]{Bing} for a brief and readable account of the origins of Brouwer's result. Some years later, Schauder conjectured that all non-empty, compact, and convex subspaces of topological vector spaces have the fixed point property, but he could only prove it for Banach spaces, see~\cite{Schauder}. Not long thereafter, Tychonoff, in~\cite{Tychonoff}, generalised this result to locally convex topological vector spaces. Only rather recently however, in a truly remarkable work~\cite{Cauty}, Robert Cauty proved Schauder's conjecture in full generality, so that we are now able to record:

\begin{example}
All non-empty, compact, convex subspaces of topological vector spaces have the fixed point property.
\end{example}

Contrary to what the preceding examples might suggest, there are also many interesting topological spaces which do not have the fixed point property.

\begin{example}
None of the spheres~$S^{n-1}=\{x\in\RR^n|\|x\|=1\}$ for~$n\geqslant0$ is a fixed point space. 
In particular, of course, the empty set ($n=0$) is not a fixed point set. The spheres are in some sense the opposite of the contractibles. The first example of a compact contractible space which does not have the fixed point property is due to~Kinoshita, see~\cite{Kinoshita}.
\end{example}

Evidently, when~$X$ and~$Y$ are homeomorphic, and~$Y$ is a fixed point space, then so is~$X$. Thus, being a fixed point space is a homeomorphism invariant. It is not a homotopy invariant, as not all contractible spaces are fixed point spaces, while singleton spaces are.

We reproduce a well-known proposition and its corollaries for the purpose of later generalization.

\begin{proposition}\label{prop:retracts}
All topological spaces~$X$ that are retracts of a fixed point space~$Y$, in the sense that there are two continuous maps~$s\colon X\to Y$ and~$r\colon Y\to X$ such that~$rs=\id_X$, are fixed point space as well.
\end{proposition}

\begin{corollary}
Fixed point spaces are connected.
\end{corollary}

\begin{corollary}\label{cor:products}
For any product~$X\times Y$ that is a fixed point space, also both factors~$X$ and~$Y$ are fixed point spaces.
\end{corollary}

The converse does not hold. The paper~\cite{Husseini} contains examples where~$X$ and~$Y$ are manifolds. See also~\cite{Brown} and the references therein for other examples.

We close this section with a result on separation properties of fixed point spaces. Fixed point spaces need not be Hausdorff spaces, but they necessarily have to satisfy some weaker separation axiom, as the next proposition shows. Recall that a topological space is a {\it Kolmogorov space} if and only if for every pair of distinct points, at least one of them has an open neighbourhood which does not contain the other.

\begin{proposition}\label{prop:Kolmogorov}
All topological spaces that have the fixed point property are Kolmogorov spaces.
\end{proposition}

\begin{proof}
By definition, if~$X$ is not a Kolmogorov space, then it contains two points~$x\not=x'$ such that all open subsets of~$X$ contain either both of them or none of them. This means that the self-map~$f\colon X\to X$ which sends~$x$ to~$x'$ and everything else to~$x$ is continuous. And~$f$ does not have a fixed point.
\end{proof}

%%%

\section{Continuous families of fixed points}\label{sec:families}

We are now addressing the question that forces itself on us: How do the fixed points of a continuous self-map depend upon the given self-map? The following is a well-known positive example, where the fixed points also vary continuously.

\begin{example}\label{ex:Banach}
Let~$(X,d)$ be a non-empty complete metric space, and let~$K$ be a real constant with~$0\leqslant K<1$. A self-map~$f\colon X\to X$ that satisfies~\hbox{$d(f(x),f(y))\leqslant Kd(x,y)$} for all~$x$ and~$y$ is called a~{\it$K$-contraction} on~$X$. The Banach fixed point theorem~\cite{Banach}, in its usual form, says that every~$K$-contraction~$f$ on~$X$ has a unique fixed point~\hbox{$p=p(f)$} in~$X$. Even more is true: The fixed points depend continuously on~$f$! More precisely, if~$(f_n)$ is a sequence of~$K$-contractions that converges uniformly to a~$K$-contraction~$f$, then the fixed points~$p_n=p(f_n)$ converge to~$p=p(f)$. In fact, the triangle equality implies
\begin{align*}
d(p_n,p)
&=d(f_n(p_n),f(p))\\
&\leqslant d(f_n(p_n),f_n(p))+d(f_n(p),f(p))\\
&\leqslant Kd(p_n,p)+d(f_n(p),f(p)),
\end{align*}
and this leads to the inequality~$d(p_n,p)\leqslant(1-K)^{-1}d(f_n,f)$, so that the uniform convergence~\hbox{$d(f_n,f)\to0$} of the self-maps implies the convergence~$d(p_n,p)\to0$ of the fixed points.
\end{example}

The Banach fixed point theorem evidently applies to metric spaces and contractions only. In general, sequences and their limits will have to be replaced by general topological concepts. We will now begin to make our set-up more precise.

\begin{definition}\label{def:cfs-m}
{\bf (Continuous family of self-maps)}\label{def:continuous_family_of_self-maps} Let~$X$ be a topological space. A {\it continuous family of self-maps} of~$X$ is a continuous map
\[
f\colon T\times X\longrightarrow X,
\]
where~$T$ is another topological space. 
\end{definition}

\begin{remark}\label{rem:bundles}
Equivalently, a continuous family of self-maps of~$X$ is the second component of a continuous self-map~$(\id_T,f)$ of~$T\times X$ which respects the projection to~$T$, the first one necessarily being the identity on~$T$. This point of view has the advantage of being closer to the idea of a self-map, but the disadvantage of carrying the redundant component~$\id_T$ around. However, this alternative suggests, as a generalization, to consider self-maps over~$T$ of total spaces~$E$ of fibre bundles~$E\to T$ with fibre~$X$ over~$T$. We do not pursue this idea here, as already the product situation turns out to be rather restrictive on~$X$. 
\end{remark}

A continuous family~$f\colon T\times X\to X$ of self-maps defines for each~$t$ in~$T$ a self-map
\[
f_t\colon X\longrightarrow X,\,x\longmapsto f(t,x),
\]
on~$X$ which is continuous. We use this notation throughout this text.

%%%

\begin{definition}{\bf (Continuous family of fixed points)}\label{def:continuous_fixed_point}
Let~$f\colon T\times X\to X$ be a continuous family of self-maps of~$X$. A {\it continuous family of fixed points} of~$f$ is a continuous map
\[
p\colon T\longrightarrow X
\]
such that
\begin{equation}\label{eq:cont_fam_of_fp}
f(t,p(t))=p(t)
\end{equation}
for all~$t$ in~$T$.
\end{definition}

The condition~\eqref{eq:cont_fam_of_fp} means that~$p(t)$ is a fixed point of~$f_t$ for all~$t$ in~$T$.

\begin{remark}
In the spirit of Remark~\ref{rem:bundles}, a continuous family~$p$ of fixed points of~$f$ is the second component of a continuous section~$(\id_T,p)$ of the projection~$\pr_T$ to~$T$ such that~$(\id_T,f)(\id_T,p)=(\id_T,p)$. Again, the first component necessarily is the identity on~$T$.
\end{remark}

%%%

\begin{example}
Just as every point of a topological space is a fixed point of some continuous self-map, for example of the constant self-map, or of the identity, so is every continuous map~\hbox{$p\colon T\to X$} a continuous family of fixed points for some continuous family~\hbox{$f\colon T\times X\to X$} of self-maps. For example, take the family~$f$ defined by~\hbox{$f(t,x)=p(t)$}, which is the family of constant self-maps, of course, or~$f(t,x)=x$, the constant family where every map~$f_t$ is the identity~$\id_X$. 
\end{example}

Example~\ref{ex:counterexample} shows that continuous families of fixed points need not exist in general.

%%%

\begin{definition}\label{def:ffpT}
{\bf (Fixed point property with respect to a topological space)} A topological space~$X$ has the {\it fixed point property with respect to a topological space}~$T$ if for all continuous families~$f\colon T\times X\to X$ of self-maps of~$X$, there exists at least one continuous family~$p\colon T\to X$ of fixed points. 
\end{definition}

\begin{examples}
Trivially, if~$T$ is empty, then all topological spaces~$X$ have the fixed point property with respect to~$T$. A topological spaces~$X$ has the fixed point property with respect to a singleton if and only if it is a fixed point space. More generally, the same statement holds true when we replace the singleton by any non-empty discrete space.
\end{examples}

In the rest of this text, we  consider the cases when~$X$ has the fixed point property with respect to~$T=I$, the unit interval, and the case when~$X$ has the fixed point property with respect to all topological spaces, respectively. Before we do so, we prove some helpful results on retracts.

%%%

\section{Retracts}

In this short section, we prove that the fixed point property of a topological space~$X$   with respect to a topological space~$T$ behaves well under retracts in~$T$ and in~$X$.

\begin{proposition}\label{prop:ST_retracts}
If a topological space~$X$ has the fixed point property with respect to a topological space~$T$, and~$S$ is a retract of~$T$, then~$X$ also has the fixed point property with respect to~$S$. 
\end{proposition}

\begin{proof}
Choose continuous maps~$j\colon S\to T$ and~$r\colon T\to S$ that satisfy the condition~\hbox{$rj=\id_S$}. If~$f\colon S\times X\to X$ is a continuous family of self-maps of~$X$ parametrized by~$S$ then consider the continuous family~$g=f(r\times\id_X)$ of self-maps of~$X$ parametrized by~$T$. By hypothesis of the proposition, it has a continuous family~\hbox{$q\colon T\to X$} of fixed points. It is now easy to check that~$p=qj$ a continuous family of fixed points for~$f$.
\end{proof}

Proposition~\ref{prop:retracts} and its corollaries transfer to the present context as follows.

\begin{proposition}\label{prop:XY_retracts}
If a topological space~$X$ is a retract of a topological space~$Y$ which has the fixed point property with respect to a topological space~$T$, then~$X$ also has the fixed point property with respect to~$T$. 
\end{proposition}

\begin{proof}
If the topological space~$X$ is a retract of the topological space~$Y$, then we choose continuous maps~\hbox{$s\colon X\to Y$} and~$r\colon Y\to X$ that satisfy the condition~$rs=\id_X$. Let~\hbox{$f\colon T\times X\to X$} be a continuous family of self-maps of~$X$. We now consider the continuous family~$g=sf(\id_T\times r)$ of self-maps of~$Y$. By hypothesis, it has a continuous family~$q$ of fixed points. This means
\begin{equation}\label{eq:XY_retracts}
sf(t,rq(t))=q(t)
\end{equation} 
for all~$t$ in~$T$. It is then easy to check that~$p=rq$ is a continuous family of fixed points of~$f$: Apply~$r$ to~\eqref{eq:XY_retracts} and use the retraction property~\hbox{$rs=\id_X$}.
\end{proof}

%%%

\section{The homotopy fixed point property}\label{sec:homotopy}

In this section we discuss the behavior of fixed points with respect to homotopies.

\begin{definition}\label{def:hfpp}{\bf (Homotopy fixed point property)} We say that a topological space~$X$ has the {\it homotopy fixed point property} if it has the fixed point property with respect to the unit interval~\hbox{$T=I=[0,1]$}.
\end{definition}

This means that a topological space~$X$ has the homotopy fixed point property if for all homotopies~\hbox{$f\colon I\times X\to X$} in~$X$ there is a continuous path~$p\colon I\to X$ such that~$p(t)$ is a fixed point of~$f_t$ for all~$t$ in~$I$.

\begin{remark}
We take care not to call a space which has the homotopy fixed point property a {\it homotopy fixed point space}, as this terminology conflicts with a different idea with the same name.
\end{remark}

There are many other examples of topological spaces which have the homotopy fixed point property, even non contractible ones. But, for the rest of this section we focus on counterexamples.

\begin{proposition}
All topological spaces that have the homotopy fixed point property are fixed point spaces. 
\end{proposition}

\begin{proof}
This result is an immediate consequence of Proposition~\ref{prop:ST_retracts}, since a singleton is a retract of the unit interval.
\end{proof}

The converse does not hold: the homotopy fixed point property is more restrictive than the ordinary fixed point property, as this example shows:

\begin{example}\label{ex:counterexample}
Continuous families of fixed points need not exist. Let~$X$ be the fixed point space~$[0,3]$, and let~$T$ be the space~$[0,2]$, both homeomorphic to the unit interval. Then
\[
f_t(x)=
\begin{cases}
1 & x\leqslant t\\
x & t\leqslant x\leqslant t+1\\
2 & t+1\leqslant x\\
\end{cases}
\]
is a continuous family of self-maps.

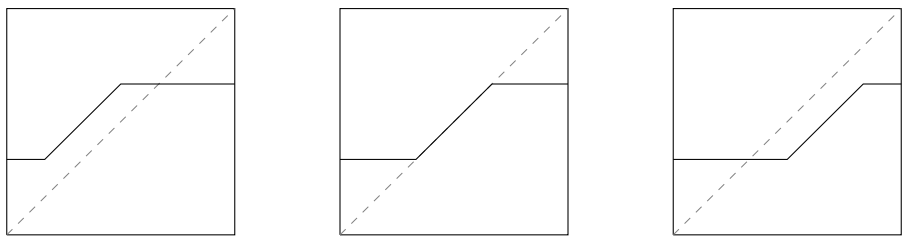
\begin{figure}[H]
\caption{Graphs of the functions~$f_t$ for~$t=1/2$,~$t=1$, and~$t=3/2$}
\begin{center}
\begin{tikzpicture}
\draw (0,0) -- (3,0) -- (3,3) -- (0,3) -- (0,0) ;
\draw[very thin,color=gray,dashed] (0,0) -- (3,3) ;
\draw (0,1) -- (0.5,1) -- (1.5,2) -- (3,2) ;
\end{tikzpicture}\hspace{3em}
\begin{tikzpicture}
\draw (0,0) -- (3,0) -- (3,3) -- (0,3) -- (0,0) ;
\draw[very thin,color=gray,dashed] (0,0) -- (3,3) ;
\draw (0,1) -- (1,1) -- (2,2) -- (3,2) ;
\end{tikzpicture}\hspace{3em}
\begin{tikzpicture}
\draw (0,0) -- (3,0) -- (3,3) -- (0,3) -- (0,0) ;
\draw[very thin,color=gray,dashed] (0,0) -- (3,3) ;
\draw (0,1) -- (1.5,1) -- (2.5,2) -- (3,2) ;
\end{tikzpicture}
\end{center}
\end{figure}

The fixed point sets of the maps~$f_t$ are 
\[
\{x\in[0,3]\,|\,f_t(x)=x\}=
\begin{cases}
\{2\} & t<1\\
[1,2] & t=1\\
\{1\} & t>1.\\
\end{cases}
\]
This makes it clear that there is no continuous family of fixed points. 
\end{example}

The preceding example shows that the unit interval, which is a fixed point space, does not have the homotopy fixed point property. More generally, the~$n$-disk~$D^n$ does not have the homotopy fixed point property as soon as~$n\geqslant 1$: If~$D^n$, for some~$n\geqslant1$, had the homotopy fixed point property, so had~$D^1$, as~$D^1$ is a retract of~$D^n$. Still more general is this theorem:

%%%

\begin{theorem}\label{thm:Banach_homotopy}
A convex subspace~$X$ of a Banach space~$V$ has the homotopy fixed point property if and only if it is a singleton.
\end{theorem}

\begin{proof}
The statement is clear for the singleton. So it is if~$X$ is empty. If~$X$ has at least two points, let~$s\colon I\to V$ be the straight path which connects the two points. The Hahn-Banach theorem provides for~a map~$r\colon V\to I$ such that the composition of the restriction of~$r$ to~$X$ with~$s$ is the identity, so that~$I$ is a retract of~$X$. Since~$I$ does not have the homotopy fixed point property by the preceding Example, nether does~$X$ by Proposition~\ref{prop:XY_retracts}.
\end{proof}

\begin{example}
Recall that one model for the Hilbert cube is the set of sequen\-ces~$(x_n)$ of real numbers such that~$|x_n|\leqslant1/n$ for all~$n$. This makes it clear that it is a convex subset of the Hilbert space~$\ell^2$. It is homeomorphic to the countably infinite (Tychonoff) product of compact intervals, hence compact. By the preceding theorem, the Hilbert cube does not have the homotopy fixed point property.
\end{example}

It is unknown whether Theorem~\ref{thm:Banach_homotopy} generalizes to arbitrary topological vector spaces.

With a little technique, we are able to implant the previously constructed counterexamples in most finite-dimensional topological spaces. Recall that for~$n\geqslant0$, the~$n$-ball is the subspace~\hbox{$B^n=\{x\in\RR^n|\|x\|<1\}$} of euclidean space~$\RR^n$.

\begin{proposition}\label{prop:emb}
All topological spaces~$X$ admitting an open embedding~\hbox{$B^n\to X$} for some~$n\geqslant 1$ do not have the homotopy fixed point property.
\end{proposition}

\begin{proof}
By Propositions~\ref{prop:XY_retracts} and Theorem~\ref{thm:Banach_homotopy}, it suffices to show that~$D^n$ is a retract of the topological space~$X$. In order to prove this, we rescale the given open embedding so as to obtain another open embedding
\[
e\colon\{x\in\RR^n\,|\,\|x\|<3\}\longrightarrow X.
\]
Consider the continuous map~$\lambda\colon[0,3]\to[0,1]$ with
\[
\lambda(t)=
\begin{cases}
t & t\leqslant1\\
2-t & 1\leqslant t\leqslant 2\\
0 & 2\leqslant t.
\end{cases}
\]
Then the continuous map
\[
f\colon\{x\in\RR^n\,|\,\|x\|<3\}\longrightarrow D^n,\,x\longmapsto\lambda(\|x\|)x
\]
extends by zero over~$e$ to a continuous map~$r\colon X\to D^n$ which restricts to the identity on~$D^n$.
\end{proof}

\begin{corollary}\label{cor:manifolds_and_CW_homotopy}
Finite-dimensional topological manifolds and finite-dimen\-sio\-nal~CW complexs do not have the homotopy fixed point property, unless they are singletons.
\end{corollary}

\begin{proof}
If the dimension of a finite-dimensional topological manifold or a finite-dimen\-sio\-nal CW complex~$X$ is positive, then it admits an open embedding~\hbox{$B^n\to X$} as required by the preceding proposition. If the dimension is zero, then~$X$ is discrete, and a discrete space which has the homotopy fixed point property must be a singleton. 
\end{proof}

%%%

\section{The universal fixed point property}\label{sec:universal}

In this section we define the class of spaces where there is a continuous choice of fixed points whenever the self-map varies continuously.

\begin{definition}\label{def:ufpp}{\bf (Universal fixed point property)} We say that a topological space~$X$ has the {\it universal fixed point property} if it has the fixed point property with respect to all topological spaces~$T$.
\end{definition}

Clearly, the universal fixed point property implies the homotopy fixed point property, which implies the fixed point property. It is therefore a strong condition for a topological space to be a universal fixed point space. For example, we already know that a convex subspace~$X$ of a Banach space~$V$ does not have the universal fixed point property unless it is a singleton~(Theorem~\ref{thm:Banach_homotopy}). If~$X$ is a finite-dimensional topological manifold, or if~$X$ is a finite-dimensional~CW complex, then~$X$ is not a universal fixed point space, unless~$X$ is a singleton.~(Corollary~\ref{cor:manifolds_and_CW_homotopy}). 

In contrast to what these classes of non-examples suggest, singletons are not the only universal fixed point spaces, and there are even non-contractible examples. We identify these examples  using the concept of selection maps, to which we turn now.

%%%

\section{A brief review of mapping spaces}\label{sec:mapping space}

As we need some results on mapping spaces, let us briefly review these. See~\cite{Munkres} and~\cite{Laures+Szymik}, for example.

%%%

Let~$X$ and~$Y$ be topological spaces. The subsets
\[
\{f\colon X\to Y\,|\,f(K)\subseteq V\},
\]
for~$K\subseteq X$ compact, and~$V\subseteq Y$ open, generate a topology on the set~$C(X,Y)$ of continuous maps~$X\to Y$, the compact-open-topology. If~$X$ is a compact Hausdorff space, and~$Y$ is metric, then this is the topology of uniform convergence. Some useful properties of the compact-open-topology in general are as follows: A continuous map
\[
f\colon T\times X\longrightarrow Y
\]
defines for each~$t$ in~$T$ a continuous map
\[
f_t\colon X\longrightarrow Y,\,x\longmapsto f(t,x),
\]
from~$X$ to~$Y$. This yields a map
\[
f^{\#}\colon T\longrightarrow C(X,Y),\,t\longmapsto f_t,
\]
into the set~$C(X,Y)$, the so-called {\it adjoint} of~$f$, which is continuous when we equip the target with the compact-open-topology. The map
\[
C(T\times X,Y)\longrightarrow C(T,C(X,Y)),\,f\mapsto f^{\#},
\]
itself is clearly injective, but it is only surjective, for example, if~$X$ is a locally compact Hausdorff space. In particular, the evaluation map
\[
C(X,Y)\times X\longrightarrow Y,\, (f,x)\longmapsto f(x)
\]
is continuous when~$X$ is a locally compact Hausdorff space.

For every continuous family~$f\colon T\times X\to X$ of self-maps of~$X$ such that~$T$ is a locally compact Hausdorff space, there is a continuous self-map
\begin{equation}\label{eq:left_alternative}
C(T,X)\longrightarrow C(T,X),\,p\longmapsto(t\mapsto f(t,p(t))).
\end{equation}

This proposition is useful since it reduces questions about continuous families of fixed points to question about (ordinary) fixed points:

\begin{proposition}\label{prop:mapping_spaces}
If~$T$ is topological space which is a locally compact Hausdorff space, then a continuous family \hbox{$f\colon T\times X\to X$} of self-maps of~$X$ has a continuous family of fixed points if and only if the self-map~\eqref{eq:left_alternative} of the mapping space~$C(T,X)$ has a fixed point.
\end{proposition}

\begin{proof}
This result follows immediately from spelling out the definitions. If the self-map~\eqref{eq:left_alternative} has a fixed point~$p$, then~$p\colon T\to X$ is a continuous map such that~\hbox{$f(t,p(t))=p(t)$}, and this is Definition~\ref{def:continuous_fixed_point} of a continuous family of fixed points of~$f$. And conversely, a continuous family of fixed points is a fixed point of that self-map.
\end{proof}

\begin{corollary}\label{cor:fpp_implies_cfp}
For all locally compact Hausdorff spaces~$T$ such that the mapping space~$C(T,X)$ is a fixed point space, the space~$X$ has the fixed point property with respect to~$T$.
\end{corollary}

Not all self-maps of~$C(T,X)$ need to have the form~\eqref{eq:left_alternative}, so that we cannot infer here that the converse of the corollary also holds.

\begin{example}
For the unit interval~$I=[0,1]$, the space~$C(I,I)$ of self-maps is not a fixed point space. This claim follows immediately from Corollary~\ref{cor:fpp_implies_cfp} and Example~\ref{ex:counterexample}. But, of course, it is also easy to see this explicitly. For example, the continuous self-map~$f\colon C(I,I)\to C(I,I)$ such that~$f(v)\colon t\mapsto tv(t)$ does not have a fixed point.
\end{example}

It is useful to know that if~$X$ is a retract of~$Y$, then~$C(X,X)$ is a retract of~$C(Y,Y)$. We  omit the proof, as we do not need this result.

%%%

\section{Selection maps}

In this section, we give a criterion for topological spaces to have the universal fixed point property. Here is the definition that lays the basis for it.

\begin{definition}
Let~$X$ be a topological space. We call a continuous map
\[
\Phi\colon C(X,X)\longrightarrow X
\]
with the property that~$\Phi(f)$ is a fixed point of~$f$,
\[
f(\Phi(f))=\Phi(f),
\] 
for all~$f$ in~$C(X,X)$, a {\it selection map} for~$X$.
\end{definition}

A selection map for~$X$ continuously selects a fixed point~$\Phi(f)$ for every self-map~$f$ of~$X$. A similar concept has been introduced (much earlier, but independently) in a related context in~\cite{products}.

\begin{proposition}\label{prop:products}
Let~$X$ be a topological space that admits a selection map. Then the product~$X\times Y$ is a fixed point space for all fixed point spaces~$Y$, and it admits a selection map for all~$Y$ that admit selection maps.
\end{proposition}

This result corresponds to and generalizes Theorem 1 in {\it loc.~cit.}. Before we begin with the proof, let us introduce some notation in connection with a continuous self-map
\[
f\colon X\times Y\longrightarrow X\times Y
\]
on a space that is a product of two other spaces~$X$ and~$Y$. For given~$x$ in~$X$ and~$y$ in~$Y$, there are continuous maps
\[
f_{X,y}\colon X\longrightarrow X, x\longmapsto\pr_Xf(x,y),
\]
and similarly~$f_{x,Y}\colon Y\to Y$. The pair~$(x,y)$ is a fixed point of~$f$ if and only if both~$x$ is a fixed point of~$f_{X,y}$ and~$y$ is a fixed point of~$f_{x,Y}$. The two maps depend continuously on the given points, and this defines maps
\[
f_{X,?}\colon Y\longrightarrow C(X,X), y\longmapsto f_{X,y},
\]
and similarly~$f_{?,Y}\colon X\to C(Y,Y)$.

\begin{proof}
We choose a selection map~$\Phi\colon C(X,X)\to X$ for~$X$.

Let~$f\colon X\times Y\to X\times Y$ be a continuous self-map. We have to show that it has a fixed point. To do so, consider the self-map
\[
f^\Phi_Y\colon Y\longrightarrow Y,y\longmapsto f_{\Phi(f_{X,y}),Y}(y).
\]
If~$Y$ is a fixed point space, then this self-map has a fixed point~$y$. Then~$x=\Phi(f_{X,y})$ is a fixed point of~$f_{X,y}$ and~$y$ is a fixed point of~$f_{x,Y}$, so that~$(x,y)$ is a fixed point of~$f$, as required.

If~$Y$ also has a selection map~$\Psi$, then we use it to define~$y=\Psi(f^\Phi_Y)$ as a function of~$f$, and produce a selection map for~$X\times Y$ from that, as in the first part of the proof.
\end{proof}

It follows from Proposition~\ref{prop:products} that topological spaces with a selection map are  fixed point spaces (take~$Y$ a singleton). But, more it true: As the main result of this section, we present a useful criterion for a topological space to have the universal fixed point property. 

\begin{theorem}\label{thm:criterion}
All topological spaces~$X$ that admits a selection map have the universal fixed point property. The converse holds for all topological spaces~$X$ for which the evaluation map~\hbox{$C(X,X)\times X\to X$} is continuous.
\end{theorem}

\begin{proof}
Let~$f\colon T\times X\to X$ be a continuous family of self-maps of~$X$ parametrized by some topological space, and let again~$f^{\#}\colon T\to C(X,X)$ denote its adjoint. Then it is easy to check that the composition~$p=\Phi f^{\#}$ is a continuous family of fixed points of~$f$.

If~$X$ is a universal fixed point space, then there exists a continuous family~$p$ of fixed points for each continuous family~$f$ of self-maps of~$X$. If the evaluation map is continuous, then we take it as~$f$, and see that it has a continuous family~$C(X,X)\to X$ of fixed points, which we denote by~$\Phi$. The fixed point equation~$f(t,p(t))=p(t)$ reads~$\ev(g,\Phi(g))=\Phi(g)$ in this case, which means~$g(\Phi(g))=\Phi(g)$ for all~$g$, as desired.
\end{proof}

\begin{corollary}
A locally compact Hausdorff space~$X$ has the universal fixed point property if and only if it has the fixed point property with respect to the space~$C(X,X)$ of self-maps of itself.
\end{corollary}

With these criteria at hand, we now proceed to present many examples of universal fixed point spaces.

%%%

\section{Spaces from partially ordered sets}

By Proposition~\ref{prop:Kolmogorov}, every fixed point space is a Kolmogorov space. In this section, we demonstrate the difficulties one encounters in the study of universal fixed point spaces by looking at examples with only finitely many points.

The category of finite Kolmogorov spaces and continuous maps is isomorphic to the category of finite partially ordered sets and monotone maps, and the isomorphism is the identity on underlying sets. See~\cite{Birkhoff:Third},  and~\cite[Proposition~7]{Stong}. Briefly, if~$X$ is a Kolmogorov space, and~$x$ is an element in~$X$, define~$U(x)$ to be the intersection of the open subsets of~$X$ which contain~$x$. When we write~$x\leqslant x'$ for~$U(x)\subseteq U(x')$, then this defines a partial order~$\leqslant$ on~$X$, and~\hbox{$U(x')=\{x\in X\,|\,x\leqslant x'\}$}. Conversely, if~$\leqslant$ is a partial order on~$X$, then we define the Kolmogorov topology on~$X$ in such a way that the open subsets~$U$ are those subsets which satisfy the property: if~$x\leqslant x'$ and~$U$ contains~$x'$, then it contains~$x$ as well.

%%%

In the context of posets, selection maps have been introduced in~\cite{products}. We now record that this concept precisely captures the universal fixed point property for the corresponding class of spaces.

\begin{proposition}\label{prop:poset_char}
A finite poset admits a selection map if and only if its associated Kolmogorov space is a universal fixed point space.
\end{proposition}

\begin{proof}
This follows from Theorem~\ref{thm:criterion}. The compact-open-topology on the set~$C(X,X)$ of continuous self-maps is the Kolmogorov topology associated with the point-wise partial order on maps:~\hbox{$f\leqslant g$} if and only if~\hbox{$f(x)\leqslant g(x)$} for all~$x$ in~$X$. See~\cite[Proposition 9]{Stong}. With respect to this order, evaluation is order preserving, hence continuous.
\end{proof}

The existence of selection maps has been shown to be indeed stronger than the fixed point property in~\cite{strong}. From our perspective, this comes as no surprise, since there are examples of fixed point spaces that are not universal. We now present examples of universal fixed point spaces that are not contractible.

%%%

Rival, in~\cite{Rival}, has introduced a notion of dismantlability (by irreducibles) for finite posets, and this has been extended to arbitrary posets in~\cite{Baclawski+Bjorner}. We do not recall the definitions here, but rather record a characterization, which follows from earlier results in~\cite{Stong}.

\begin{proposition} 
A finite poset is dismantlable if and only if its associated Kolmogorov space is contractible. 
\end{proposition}

All dismantlable posets admit a selection map, see~\cite{products}. Or, all contractible finite Kolmogorov spaces are universal fixed point spaces. For example, the Sierpi\'nski space~$S=\{\alpha,\omega\}$ with open subsets~$\emptyset$,~$\{\omega\}$, and~$S$ is a universal fixed point space, as is any Kolmogorov space associated with a finite poset that has a minimal or maximal element. However, the converse is not true:  Not all finite Kolmogorov spaces that are universal fixed point spaces are also contractible. Here is an example.

\begin{example}\label{ex:not_contractible}
Rival, also in~\cite{Rival}, has given an example of a finite poset with the fixed point property that is not dismantlable, so that its associated Kolmogorov space is not contractible. The list in~\cite{fpp_small_sets} contains some more examples. All these examples even admit selection maps by~\cite{strong_fpp_small_sets}. 
\tikzstyle{dot}=[circle,draw,fill,scale=0.5]
\tikzstyle{line}=[-,very thick]

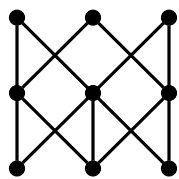
\begin{figure}[H]
\caption{A poset with a selection map that is not dismantlable}
\begin{center}
\begin{tikzpicture}
\node[dot] (1) at (0,0) {};
\node[dot] (2) at (0,1) {};
\node[dot] (3) at (0,2) {};
\node[dot] (4) at (1,0) {};
\node[dot] (5) at (1,1) {};
\node[dot] (6) at (1,2) {};
\node[dot] (7) at (2,0) {};
\node[dot] (8) at (2,1) {};
\node[dot] (9) at (2,2) {};

\draw[line] (1) to (3) ;
\draw[line] (7) to (9) ;
\draw[line] (1) to (9) ;
\draw[line] (7) to (3) ;
\draw[line] (2) to (4) ;
\draw[line] (2) to (6) ;
\draw[line] (8) to (4) ;
\draw[line] (8) to (6) ;
\draw[line] (4) to (5) ;
\end{tikzpicture}
\end{center}

\end{figure}Therefore, the associated Kolmogorov spaces have the universal fixed point property by Proposition~\ref{prop:poset_char}
\end{example}

%%%

\section*{Acknowledgments}

This research has been supported by the Danish National Research Foundation through the Centre for Symmetry and Deformation (DNRF92). I also thank Michal Kukiela for a helpful correspondence that brought the references in Order to my attention.

%%%

%%%

\vfill

\parbox{\linewidth}{%
{\it Affiliation:}\\
Mathematisches Institut\\Heinrich-Heine-Universit\"at\\40225 D\"usseldorf\\Germany\\\phantom{ }

{\it Current address:}\\
Department of Mathematical Sciences\\University of Copenhagen\\2100 Copenhagen~\O\\Denmark\\\phantom{ }

{\it Email address:}\\
szymik@math.ku.dk
}

\end{document}